\documentclass[a4paper,11pt,reqno]{amsart}
\usepackage{amsmath, amssymb, amsthm}
\usepackage{hyperref}
\usepackage[in]{fullpage}
\usepackage{cite}	
\usepackage{hyperref,xcolor}
\usepackage{fancyhdr,verbatim}

\numberwithin{equation}{section}

\definecolor{dark-green}{rgb}{0.1,0.1,0.3}

\hypersetup{
colorlinks=true,
    pdfborder={0 0 0},
citecolor=dark-green,
linkcolor=blue
}

\newcommand{\floor}[1]{\left\lfloor #1 \right\rfloor}

\newcommand{\ab}[1]{\left\lvert #1 \right\rvert}

\newcommand{\md}[1]{\left(\text{mod}\ #1\right)}

\newcommand{\sg}{\sigma}

\newcommand{\al}{\alpha}

\newcommand{\eps}{\epsilon}

\newcommand{\bt}{\beta}

\newcommand{\vp}{\varphi}

\newcommand{\lm}{\limits}

\newcommand{\st}{\substack}

\newcommand{\cg}{\equiv}

\newcommand{\mc}{\mathcal}

\newcommand{\fr}{\frac}

\newcommand{\tps}{\texorpdfstring}

\newtheorem{thm}{Theorem}[section]
\newtheorem{thmA}{Theorem}

\newtheorem{cor}[thm]{Corollary}
\newtheorem{corA}{Corollary}

\newtheorem{lem}[thm]{Lemma}

\theoremstyle{definition}

\newtheorem{defn}[thm]{Definition}

\newtheorem{rem}[thm]{Remark}

\begin{document}

\title{On correlations of certain multiplicative functions}
\author{R. Balasubramanian}
\address{Department of Mathematics\\
  Institute of Mathematical Sciences\\
  Chennai, India-600113}
\email[R. Balasubramanian]{balu@imsc.res.in}
 \author{Sumit Giri}
\address{School of Mathematics, Tel Aviv University\\ P.O.B. 39040, Ramat Aviv\\ Tel Aviv 69978, Israel.}
\email[Sumit Giri]{sumit.giri199@gmail.com}
\author{Priyamvad Srivastav}
\address{Department of Mathematics\\
  Institute of Mathematical Sciences\\
  Chennai, India-600113}
\email[Priyamvad Srivastav]{priyamvads@imsc.res.in}
  \date{}
\thispagestyle{empty}

\begin{abstract}
In this paper, we study sums of shifted products $\sum\lm_{n \leq x} F(n) G(n-h)$ for any $|h| \leq x/2$ and arithmetic functions $F=f*1$ and $G=g*1$, with $f$ and $g$ small. We obtain asymptotic formula for different orders of magnitude of $f$ and
$g$. We also provide asymptotic formula for sums of the type $\sum\lm_{n \leq x} \mu^2(n) G(n-h)$, where $G=g*1$ and $g$ is small. For small order of magnitudes of $f$ and $g$, we improve the error terms and make them independent of $h$.
\end{abstract}

\maketitle{}

\section{Introduction}
\thispagestyle{empty}
Let $F$ and $G$ be two arithmetic functions. In \cite{7}, the first two authors studied the problem of getting an asymptotic formula for the sum $\sum\lm_{n \leq x} F(n)G(n-h)$, where $F=f*1$ and $G=g*1$, under the assumption that for primes $p$, $f(p)$ and $g(p)$ are close to $1$. In this paper, we continue the investigation (See Theorem \ref{main1} and Theorem \ref{main2}). We also show that this method is equally applicable to the asymptotic formula for $\sum\lm_{n \leq x}\mu^2(n)G(n-h)$. \\ \vskip 0.07in
In \cite{15}, Mirsky considers the general sum $\sum\lm_{n \leq x}F_1(n+k_1) \dots F_s(n+k_s)$, with $F_j=1*f_j$ and $f_j(p)=O(p^{-\sg+\eps})$ for each $j$. In \cite{3}, Stepanauskas considers $\sum\lm_{n \leq x} F(n)G(n-h)$ under the weaker assumption $\sum\lm_p \fr{f(p)+g(p)-2}{p}$ is convergent. In \cite{4}, Stepanauskas and Siaulys also consider the sum $\sum\lm_{p \leq x} F(p+1) G(p+2)$, where the sum runs over the primes. \vskip 0.06in
In \cite{16}, Coppola, Murty, Saha consider the problem of $\sum\lm_{n \leq x}F(n)G(n-h)$ under a general condition that $F$ and $G$ admit a Ramanujan expansion. \vskip 0.06in
A considerable amount of work has been done for such shifted sums. For instance, one can see papers of Carlitz\cite{13}, Choi and Schwarz\cite{5}, Katai \cite{2} and Rearick \cite{12}. \vskip 0.06in

Since all these results have been proved under different conditions, it is difficult to compare these results. However, functions like $\fr{\vp_s(n)}{n^s},\fr{\sg_s(n)}{n^s}$ are the common threads between these results and the results proved in this paper. We shall later compare these results in section \ref{compare}.

\medskip

\section{Statement of the theorems}

In \cite{7}, the following theorem was proved

\begin{thmA}
\label{BG}
Let $E_1(x)=\sum\lm_{n \leq x}|f(n)|$ and $E_2(x)=\sum\lm_{n \leq x}|g(n)|$. For $h \neq 0$, let 
$$ C(h) = \sum\lm_{(d_1,d_2) \mid h} \fr{f(d_1)g(d_2)}{[d_1,d_2]}. $$
Then 
$$ \sum\lm_{n \leq x} F(n)G(n-h) = x C(h) + O\left( h E_1(x) E_2(x) \right). $$
\end{thmA}
\medskip

Our aim is to improve the error term.
\begin{defn}
For $\al>0$, define $\mc{A}_{\al}$ to be the class of arithmetic functions $g$ satisfying $g(n)=O(n^{-\al})$ for each $n$. \\ \vskip 0.02in
\end{defn}
For ease of exposition, assume that $f \in \mc{A}_{\al}$ and $g \in A_{\bt}$ for some $0< \al \leq \bt<1$. We also assume that $F(n)$ and $G(n)$ are $0$, if $n \leq 0$. Then, Theorem \ref{BG} gives
\begin{corA}
\label{Cor1}
We have, under the conditions above,
\begin{equation*}
\sum\lm_{n \leq x} F(n)G(n-h)=x C(h) + O\left( h x^{2-\al-\bt}  \right).
\end{equation*}
\end{corA}
\bigskip
Let 
\begin{equation}
\label{Ealbtx}
E(x;\al,\bt) = \begin{cases} x^{1-\al}, & \al<\bt \text{ and }\al<1, \\ x^{1-\al} \log x, & \al=\bt<1, \\ \log x, & 1=\al<\bt, 
\\ \log^2 x, & \al=\bt=1, \\ 1, & 1<\al<\bt. \end{cases}
\end{equation}
\medskip
Then, we prove

\begin{thm}
\label{main1}
Suppose that $f \in \mc{A}_{\al}$ and $g \in \mc{A}_{\bt}$. Then, uniformly for all $h$, $|h| \leq x/2$, we have
$$ \sum\lm_{H \leq n \leq x} F(n)G(n-h) = (x-H) C(h) + O\left( E(x;\al,\bt) \right), $$
where
$$ H=\begin{cases} 1, & h \leq 0, \\ h, & h>0, \end{cases} $$
$$ C(h) = \sum\lm_{\st{a,b \geq 1 \\ (a,b) \mid h}} \fr{f(a) g(b)}{[a,b]}, $$
and the $O$-constant is absolute.
\end{thm}

\begin{rem}
Theorem \ref{main1} improves Theorem \ref{BG} in all cases(in terms of $h$) and also improves upon Corollary \ref{Cor1} in terms of $x$ and $h$. 
\end{rem}

\begin{rem}
 Theorem \ref{main1} also covers the case $h=0$. Also, since $f(a) \ll a^{-\al}$, $g(b) \ll b^{-\bt}$, it follows that $C(h)$
is well defined. If $f$ and $g$ are multiplicative, then $C(h)$ admits a product expansion
$$ C(h)=\prod\lm_p \left( \sum\lm_{\min\{e_1,e_2\} \leq v_p(h)} \fr{f(p^{e_1}) g(p^{e_2})}{p^{\max\{e_1,e_2\}}} \right). $$
\end{rem}

\medskip

The method of proof of Theorem \ref{main1} also applies to study sums of the form $\sum\lm_{n \leq x} \mu^2(n)G(n-h)$. \vskip 0.06in

Let 
\begin{equation}
\label{E1albtx}
E_1(x;\al) = \begin{cases} x^{1-\al}, & 0<\al \leq 1/2, \\ x^{1/2}, & \al>1/2. \end{cases}
\end{equation}

We prove

\begin{thm}
\label{main2}
Let $G(n)=\sum\lm_{d \mid n} g(d)$, where $g \in \mc{A}_{\al}$ for some $\al>0$. Let $\eps>0$. Then, uniformly for all $|h| \leq x/2$, we have
\begin{equation*}
\sum\lm_{n \leq x} \mu^2(n) G(n-h)=(x-H) K(h) + O\left( x^{\eps} E_1(x;\al) \right),
\end{equation*}
where
$$ K(h) = \sum\lm_{\st{a,b \geq 1 \\ (a^2,b) \mid h}} \fr{\mu(a) g(b)}{[a^2,b]},$$
and $H$ is as defined in Theorem \ref{main1}. 
\end{thm}

\begin{rem}
In the appendix, we shall remark how to remove the $x^{\eps}$ from the error term when $\al$ is not in a neighborhood of $1/2$. 
\end{rem}

\begin{rem}
Theorem \ref{main2} also covers the case $h=0$. Also, $K(h)$ is well-defined because of the condition $g \in A_{\al}$. Again, if $g$ is multiplicative, then $K(h)$ admits a product expansion
$$ K(h) = \prod\lm_p \left( \sum\lm_{\max\{2e_1,e_2\} \leq v_p(h)} \fr{\mu(p^{e_1}) g(p^{e_2}) }{p^{\max\{2e_1,e_2\}}} \right).$$
\end{rem}

\medskip

By taking $G(n)=\fr{\vp(n)}{n}$, we have

\begin{cor}
\label{Cor2}
Uniformly for $|h| \leq x/2$, we have
$$ \sum\lm_{n \leq x} \mu^2(n) \fr{\vp(n-h)}{n-h} = (x - H) \prod\lm_p \left( 1-\fr{2}{p^2} \right) \prod\lm_{p \mid h} \left( 1+\fr{1}{p^3-2p} \right) + O\left( x^{1/2} \right) .$$
\end{cor}
\medskip

In particular, for $h=0$,
$$ \sum\lm_{n \leq x} \mu^2(n) \fr{\vp(n)}{n} = x \prod\lm_p\left( 1-\fr{2}{p^2} \right) \left( 1+\fr{1}{p^3-2p} \right) + O\left( x^{1/2} \right) .$$

\begin{rem}
\label{omega}
We observe that the Dirichlet series of $\mu^2(n) \fr{\vp(n)}{n}$ is
$$ \sum\lm_{n=1}^{\infty} \fr{\mu^2(n) \vp(n)}{n^{1+s}}=\fr{\zeta(s) H(s)}{\zeta(2s) \zeta(4s)}, $$
where $H(s)$ is absolutely convergent in $\Re(s) \geq 1/8$. Consequently, by Landau's theorem, the error term for the case $h=0$ in Corollary \ref{Cor2} is $\Omega(x^{1/2-\eps})$ if the zeta function were to have a zero close to $Re(s)=1$. This shows that Corollary \ref{Cor2} cannot be improved except for terms like $\exp\left( -c(\log x)^{2/5} (\log \log x)^{3/5} \right)$ unless a good zero-free region for the Riemann zeta function is assumed.
\end{rem}
\medskip
We also note that, by partial summation and using Theorem \ref{main1}, we can also write an asymptotic formula for $\sum\lm_{n \leq x}Q(n) F(n)G(n-h)$, for any function $Q(n)$ such that $Q(t)$ is differentiable for $1 \leq t \leq x$ and $Q'(t)$ is bounded in 
$1 \leq t \leq x$. It is as follows

$$   \sum\lm_{n \leq x} Q(n) F(n) G(n-h)= C(h) \int\lm_1^x Q(t) \, dt + Q(x)E(x;\al,\bt)+O\left( \int\lm_1^x |Q'(t)| |E(t;\al,\bt)| \, dt \right) .$$


\bigskip

\section{Preliminary Lemmas}
In this section, we start with some preliminary lemmas for the proof of Theorem \ref{main1}. We assume throughout $\bt \geq\al > 0$. Recall that
$$ E(x)=E(x;\al,\bt) =\begin{cases} x^{1-\al}, & \al<\bt \text{ and } \al<1, \\ x^{1-\al} \log x, & \al=\bt<1, \\ \log x, & 1=\al<\bt, 
\\ \log^2 x, & \al=\bt=1, \\ 1, & 1<\al<\bt .\end{cases} $$
The statements of the lemmas in this section hold true for all $0<\al \leq \bt$. However, we restrict the proofs only to the case $\bt>\al$ and $\al<1$. The proof works mutandis-mutandis for $\al \geq 1$ and the case $\bt=\al$ will have an extra log factor. \vskip 0.03in
When $\bt>\al$ and $\al<1$, we find that $E(x) = O(x^{1-\al})$.
\medskip

\begin{lem}\text{ }
\label{Lem1}
\begin{itemize}
\item[(a)] If $y \geq 1$, then
$$ \sum\lm_{mn \geq y} \fr{1}{m^{1+\al}n^{1+\bt}} = O\left( \fr{E(y)}{y} \right). $$
\item[(b)] If $x \geq 1$, then
$$S=\sum\lm_{[a,b] \geq x} \fr{1}{a^{\al}b^{\bt}[a,b]} = O\left( \fr{E(x)}{x} \right).$$
\end{itemize}
\end{lem}

\begin{proof}
We first prove (a). As $\bt>\al$, we get the sum to be equal to
\begin{equation*}
\begin{split}
\sum\lm_{n \geq 1} \fr{1}{n^{1+\bt}} \sum\lm_{m\geq y/n} \fr{1}{m^{1+\al}} &= \sum\lm_{n \leq y} \fr{1}{n^{1+\bt}} \sum\lm_{m \geq y/n} \fr{1}{n^{1+\al}} + \sum\lm_{n>y} \fr{1}{n^{1+\bt}}
\sum\lm_{m \geq 1} \fr{1}{m^{1+\al}}\\
&= O\left( \fr{1}{y^{\al}}\sum\lm_{n \leq y} \fr{1}{n^{1+\bt-\al}}\right) + O\left( \sum\lm_{n>y} \fr{1}{n^{1+\bt}} \right) =O(y^{-\al})
\end{split}
\end{equation*}
and hence (a). \vskip 0.03in

To prove (b), we split the sum depending upon the value of $l=\text{gcd}(a,b)$. Write $a=ml$ and $b=nl$. Then
$$ S \ll \sum\lm_{l \geq 1} \fr{1}{l^{1+\al+\bt}}\sum\lm_{mn \geq x/l} \fr{1}{m^{1+\al} n^{1+\bt}}. $$
Thus, using part (a),
\begin{equation*}
\begin{split}
S &\ll \sum\lm_{l \leq x} \fr{1}{l^{\al+\bt}}\fr{ E(x/l)}{x} + \sum\lm_{l>x} \fr{1}{l^{1+\al+\bt}} \sum\lm_{m,n \geq 1} \fr{1}{m^{1+\al} n^{1+\bt}}\\
&\ll x^{-\al} \sum\lm_{l \leq x} \fr{1}{l^{1+\bt}} + \sum\lm_{l>x} \fr{1}{l^{1+\al+\bt}} \ll x^{-\al}.
\end{split}
\end{equation*}
This completes the proof.
\end{proof}

\medskip

\begin{lem}
\text{ }
\label{Lem2}
\begin{itemize}
\item[(a)] 
If $y \geq 1$. then
$$ \sum\lm_{mn \leq y} \fr{1}{m^{\al} n^{\bt}} = O\left( E(y) \right). $$
\item[(b)]
If $x \geq 1$, then
$$ \sum\lm_{\st{[a,b] \leq x\\(a,b)=l}} \fr{1}{a^{\al}b^{\bt}} = O\left( \fr{E(x)}{l^{1+\bt}} \right).$$
\item[(c)]
If $x \geq 1$, then
$$ \sum\lm_{[a,b] \leq x} \fr{1}{a^{\al} b^{\bt}} = O \left( E(x) \right).$$
\end{itemize}
\end{lem}

\begin{proof}
We have
\begin{equation*}
\sum\lm_{mn \leq y} \fr{1}{m^{\al}n^{\bt}} = \sum\lm_{n \leq y} \fr{1}{n^{\bt}} \sum\lm_{m \leq y/n} \fr{1}{m^{\al}}= O\left( \sum\lm_{n \leq y} \fr{1}{n^{\bt}} \left( \fr{y}{n} \right)^{1-\al} \right)
=O\left(y^{1-\al}\right)
\end{equation*}
and hence (a). \vskip 0.03in
To prove (b), again split the sum depending upon the value of $l=\text{gcd}(a,b)$ and use (a). Write $a=ml$ and $b=nl$ as before. Then the given sum equals
\begin{equation*}
\fr{1}{l^{\al+\bt}}\sum\lm_{\st{mn \leq x/l\\(m,n)=1}} \fr{1}{m^{\al}n^{\bt}} \ll \fr{1}{l^{\al+\bt}}\left(\fr{x}{l}\right)^{1-\al}
\end{equation*}
and this proves (b). \vskip 0.04in
Now, (c) is obtained easily from (b).
\end{proof}

\medskip

\begin{lem}
\label{Lem3}
\text{ }
\begin{itemize}
\item[(a)] Let $y \geq 1$ and $|k| \leq y/2$. Then
$$ S_1=\sum\lm_{m \leq y} \sum\lm_{\st{a \mid m\\b \mid m-k\\ab \geq y}}a^{-\al}b^{-\bt} = O\left( E(y) \right)$$
and the $O$-constant is absolute.
\item[(b)]
Let $x \geq 1$ and $|h| \leq x/2$. Then
$$ S_2 = \sum\lm_{n \leq x} \sum\lm_{\st{c \mid n \\ d \mid n-h\\ [c,d] \geq x}} c^{-\al} d^{-\bt} = O\left( E(x) \right)$$
and the $O$-constant is absolute.
\end{itemize}
\end{lem}

\begin{proof}
To prove (a), put $m=ac$, $m-k=bd$. Then writing the sum in terms of $c$ and $d$, we have
$$ S_1=\sum\lm_{m \leq y} \sum\lm_{\st{c \mid m\\ d \mid m-k \\ cd \leq \fr{m(m-k)}{y}}} \left( \fr{m}{c} \right)^{-\al} \left( \fr{m-k}{d} \right)^{-\bt}. $$
We note that $cd \leq \fr{m(m-k)}{y} \leq m$. This implies $m \geq cd$. Thus,
$$ S_1 \ll \sum\lm_{c \leq y} c^{\al} d^{\bt} \sum\lm_{\st{m \cg 0 \md{c}\\m \cg k \md{d}\\cd \leq m \leq y}} m^{-\al} (m-k)^{-\bt}. $$
The congruence on $m$ gives $m \cg r \md{[c,d]}$. Thus the $m$-sum is at most
$$ \ll \sum\lm_{\st{m \cg r \md{[c,d]}\\cd \leq m \leq y}} m^{-\al -\bt} \ll \sum\lm_{\st{m \cg 0 \md{[c,d]}\\cd \leq m \leq 2y}} m^{-\al-\bt}. $$
Let $\text{gcd}(c,d)=l$ and write $m=j[c,d]$, with $l \leq j \leq \fr{2y}{[c,d]}$. The $m$-sum is then
$$ \ll [c,d]^{-\al-\bt} \sum\lm_{l \leq j \leq\fr{2y}{[c,d]}} j^{-\al-\bt}. $$
Hence
$$ S_1 \ll \sum\lm_{\st{l,j\\l \leq j}}j^{-\al-\bt} \sum\lm_{\st{[c,d] \leq 2y/j \\ (c,d)=l}} \fr{c^{\al} d^{\bt}}{ \ [c,d]^{\al+\bt}}. $$
The second sum above is 
$$ \sum\lm_{\st{[c,d] \leq 2y/j\\ (c,d)=l}} \fr{c^{\al} d^{\bt} l^{\al+\bt}}{(cd)^{\al+\bt}} \ll l^{\al+\bt} \sum\lm_{\st{[c,d] \leq 2y/j \\ (c,d)=l}} c^{-\bt} d^{-\al}. $$
From Lemma \ref{Lem2} (b), the above sum is
$$ = O\left( \fr{E(y/j)}{l^{1-\al}} \right). $$
For $\al<1$, this error is 
$$ O\left( \fr{y^{1-\al}}{(jl)^{1-\al} } \right). $$
Thus, 
$$ S_1 \ll y^{1-\al} \sum\lm_j \fr{1}{j^{1+\bt}} \sum\lm_{l\leq j} \fr{1}{l^{1-\al}} \ll y^{1-\al} \sum\lm_{j \leq y} \fr{1}{j^{1+\bt-\al}} $$
and this proves (a). \vskip 0.03in

Now, we prove (b) by splitting the sum into $(c,d)=l$. Write the given sum as
\begin{equation*} 
\begin{split}
S_2 &= \sum\lm_{n \leq x} \sum\lm_{\st{l \mid n \\ l \mid n-h}} \sum\lm_{\st{c \mid n \\ d \mid n-h\\ cd \geq lx \\ (c,d)=l}} c^{-\al} d^{-\bt}
=\sum\lm_{n \leq x} \sum\lm_{\st{l \mid h}} \sum\lm_{\st{lc \mid n  \\ ld \mid n-h \\ cd \geq x/l\\(c,d)=1}} (lc)^{-\al} (ld)^{-\bt}\\
&\ll \sum\lm_{l \mid h} l^{-\al-\bt} \sum\lm_{\st{n \leq x\\ n \cg 0 \md{l}}} \sum\lm_{\st{c \mid n/l \\ d \mid (n-h)/l \\ cd \geq x/l}} c^{-\al} d^{-\bt}.
\end{split}
\end{equation*}
Write $n/l=n'$ and $h/l=h'$ so that the given sum reduces to
\begin{equation*}
S_2 \ll \sum\lm_{l \mid h} l^{-\al-\bt} \sum\lm_{n' \leq x/l} \sum\lm_{\st{c \mid n' \\ d \mid n'-h' \\ cd \geq x/l}} c^{-\al} d^{-\bt}. 
\end{equation*}
Therefore, by part (a), it follows that
$$ S_2 \ll \sum\lm_{l \mid h} l^{-\al-\bt} E(x/l) \ll x^{1-\al} \sum\lm_{l \mid h} \fr{1}{l^{1+\bt}} \ll x^{1-\al}.$$
This proves (b).
\end{proof}
\medskip
Now, we give preliminaries for the proof of Theorem \ref{main2}. Recall that 
$$ E_1(x) = E_1(x;\al) = \begin{cases} x^{1-\al}, & 0<\al \leq 1/2, \\ x^{1/2}, & \al>1/2. \end{cases}$$
\begin{lem}\text{ }
\label{Lem4}
\begin{itemize}
 \item[(a)] $$ S= \sum\lm_{H \leq n \leq y} \sum\lm_{\st{ca^2  \mid n \\ b \mid n-h \\ a^2b>z}}b^{-\al} = O\left(\fr{y^{\eps}}{c} \left( \fr{y}{z} \right)^{\al} E_1(x)\right). $$
 \item[(b)] $$ \sum\lm_{H \leq n \leq x} \sum\lm_{\st{a^2 \mid n \\ b \mid n-h \\ [a^2,b]>x}} b^{-\al} = O(x^{\eps} E_1(x)). $$
\end{itemize}
\end{lem}

\begin{proof}
We fist prove (a). Observe that since $ca^2 \mid n$, we have $ca^2 \leq y$. Break the sum over $a$ and $b$ dyadic-ally i.e let $a \sim A$ and $b \sim B$. Then
\begin{equation*}
\begin{split}
S_{A,B} &=  \sum\lm_{H \leq n \leq y} \left( \sum\lm_{\st{ca^2 \mid n \\ a \sim A}} 1 \right) \left( \sum\lm_{\st{b \mid n-h \\ b \sim B}}  b^{-\al}\right) \ll B^{-\al} y^{\eps} \sum\lm_{H \leq n \leq y} \sum\lm_{\st{ca^2 \mid n\\ a \sim A}} 1 \\
&\ll y^{\eps}B^{-\al} \sum\lm_{a \sim A} \sum\lm_{\st{H \leq n \leq y \\ n \cg 0 \md{a^2c}}} 1 \ll y^{\eps} B^{-\al} \sum\lm_{a \sim A} \left( \fr{y}{ca^2} + O(1) \right) \ll \fr{y^{1+\eps}}{cAB^{\al}}.
\end{split}
\end{equation*}
Now, summing over $A$ and $B$ in geometric progressions with $A \leq y^{1/2}$, $B \leq y$ and $A^2 B >z$, we obtain the desired result. \vskip 0.02in
We now prove (b). Let $(a^2,b)=l_1^2l_2$, with $l_2$ square-free. Hence, we have $a=kl_1l_2$ and $b=ml_1^2l_2$ and $[a^2,b]=k^2 m(l_1l_2)^2$. For a fixed $l_1,l_2$, the desired sum is
$$ \sum\lm_{H \leq n \leq x} \sum\lm_{\st{k^2l_1^2 l_2^2 \mid n \\ ml_1^2l_2 \mid n-h \\ k^2m (l_1l_2)^2>x \\ l_1^2l_2 \mid h}} b^{-\al}. $$
Write $h=h' l_1^2l_2$ and $n=n' l_1^2l_2$. The given sum now becomes,
 $$ \ll \sum\lm_{\st{H \leq n \leq x \\ n \cg 0(l_1^2 l_2)}} \sum\lm_{\st{k^2 l_1^2 l_2 \mid n \\ ml_1^2 l_2 \mid n-h \\ k^2m > x/(l_1l_2)^2}} b^{-\al} 
 \ll (l_1^2 l_2)^{-\al} \sum\lm_{H/l_1^2l_2 \leq n' \leq x/(l_1^2l_2)} \sum\lm_{\st{l_2k^2 \mid n' \\ m \mid n' - h' \\ k^2m>x/(l_1l_2)^2}} m^{-\al}. $$
 Applying part (a) to the above sum with $y=\fr{x}{l_1^2l_2}$, $z=\fr{x}{(l_1l_2)^2}$ and $c=l_2$, we obtain that for a fixed $l_1,l_2$ the given sum is,
 $$ \ll (l_1^2 l_2)^{-\al} \left( \fr{x}{l_1^2 l_2} \right)^{\eps} l_2^{\al-1} E_1\left( \fr{x}{l_1^2 l_2} \right). $$
 Summing over $l_1^2l_2 \leq x$, we obtain the desired result.
\end{proof}
\begin{rem}
In the final step above, we have summed over all $l_1^2l_2 \leq x$ instead of $l_1^2 l_2 \mid h$. This shows that the $O$-constant is indeed independent of $h$.
\end{rem}

\medskip

\begin{lem}\text{ }
\label{Lem5} 
\begin{itemize}
 \item[(a)] $$\sum\lm_{a^2b \leq y} b^{-\al} = O(E_1(y)). $$
 \item[(b)] $$ \sum\lm_{[a^2,b] \leq x} b^{-\al} = O(E_1(x)). $$
\end{itemize}
\end{lem}

\begin{proof}
For (a), we follow the proof of Lemma \ref{Lem2} (a). For (b), let $(a^2,b)=l_1^2l_2$, with $l_2$ square-free. Write $a=kl_1l_2$ and $b=ml_1^2l_2$ as in the proof of Lemma \ref{Lem4} (b). The sum then reduces to the sum in part (a). Summing over $l_1^2l_2 \leq x$ gives the desired result. 
\end{proof}
\medskip

\begin{lem} \text{ }
\label{Lem7}
\begin{itemize}
 \item[(a)] $$ \sum\lm_{a^2b>y} \fr{1}{a^2 b^{1+\al}} = O\left( \fr{E_1(y)}{y} \right). $$
 \item[(b)] $$ \sum\lm_{[a^2, b] >x} \fr{b^{-\al}}{[a^2,b]} = O\left( \fr{E_1(x)}{x} \right). $$
\end{itemize}
\end{lem}

\begin{proof}
For (a), we follow the proof of Lemma \ref{Lem1} (a). For (b), let $(a^2,b)=l_1^2 l_2$ with $l_2$ square-free. Then $a=kl_1l_2$ and $b=ml_1^2l_2$. The sum then reduces to a sum of the kind in part (a). Summing over $l_1,l_2$ then gives the desired result.
\end{proof}

\medskip
\begin{defn}
Let the function $L(n)$ be defined by
\begin{equation*} 
\label{M} L(n) = \prod\lm_{p \mid n} p^{\floor{\fr{v_p(n)}{2}}}. 
\end{equation*}
In particular, if $s$ is square-free, then $L(r^2 s)=r$.
\end{defn}

\medskip
\begin{lem}
\label{solutions}
Let $a$, $m$ be positive integers, $h \neq 0$. Let $g=\text{gcd}(h,m)$. Then the equation
$$ ax^2 \cg h \md{m}, $$
has at most $L(m) \tau(m)$ solutions modulo $m$.
\end{lem}

\begin{proof}
If $(a,m)>1$, then $\text{gcd}(a,m) \mid h$. Cancelling the factor, we have
$$ ax_1^2 \cg h_1 \md{m_1}, $$
where $m_1=\fr{m}{(a,m)}$ and $(m_1,a_1)=1$. \\ \vskip 0.01in
Note that any solution of the latter equation lifts to a unique solution of $ax^2 \cg h \md{m}$. Since $(m_1,a_1)=1$, the latter equation is the same as $x^2 \cg k \md{m_1}$. Now, write $m_1=q_1q_2$, where $q_1$ is the product of prime powers $p^l$ with $v_p(m_1) \leq v_p(k)$ and $q_2$ is a product of those prime powers $p^l$ with $v_p(m_1)>v_p(k)$.\vskip 0.01in
The equation $x^2 \cg k \md{q_1}$ is the same as $x^2 \cg 0 \md{q_1}$ and has at most $L(q_1)$ solutions. The equation $x^2 \cg k \md{q_2}$ has at most $\tau(q_2)$ solutions. Thus, the total number of solutions is at most $L(q_1) \tau(q_2)$. 
Since $q_1 \mid m$, we get $L(q_1) \leq L(m)$. Since $\tau(q_2) \leq \tau(m)$, we are through.  
\end{proof}

\bigskip

\section{Proof of Theorem \tps{\ref{main1}}{\ref{main1}}}
Now, we prove Theorem \ref{main1}. We have
\begin{equation*}
S= \sum\lm_{H \leq n \leq x} \sum\lm_{\st{a \mid n \\ b \mid n-h}} f(a) g(b) = \sum\lm_{H \leq n \leq x} \sum\lm_{[a,b] \leq x} f(a) g(b) + \sum\lm_{H \leq n \leq x} \sum\lm_{[a,b]>x} f(a) g(b).
\end{equation*}
The second term on the rightmost side above is $O(E(x))$ by Lemma \ref{Lem3} (b). The first term is
$$ \sum\lm_{[a,b] \leq x} f(a)g(b) \sum\lm_{\st{H \leq n \leq x\\ n \cg 0\md{a}\\ n \cg h \md{b}}} 1 = \sum\lm_{\st{[a,b] \leq x\\ (a,b) \mid h}} f(a) g(b) \left( \fr{x-H}{[a,b]} + O(1) \right) $$
and the $O$-term is $O\left( E(x) \right)$ by Lemma \ref{Lem2} (c). The main term is
$$ (x-H) \sum\lm_{(a,b) \mid h} \fr{f(a)g(b)}{[a,b]} - (x-H) \sum\lm_{\st{(a,b) \mid h \\ [a,b]>x}} \fr{f(a)g(b)}{[a,b]}.$$
The first term is $(x-H) C(h)$ and the second term is $O(E(x))$ by Lemma \ref{Lem1} (b).
\bigskip
\section{Comparison with earlier results}
\label{compare}
Now, we make comparison of our results with earlier results. \\ \vskip 0.03in
In Theorem \ref{main1}, we take $F(n)=\fr{n}{\vp(n)}$ and $G(n)=\fr{\sg(n)}{n}$. In this case, $f(p)=\fr{1}{p-1}$, $f(p^{\al})=0$ for $\al \geq 2$ and $g(n)=1/n$. Hence, one can take $\al=1-\eps$ for any $\eps>0$ and $\bt=1$. This gives by Theorem \ref{main1},

\begin{cor} \text{ }
\begin{itemize}
 \item[(a)] $$ \sum\lm_{n \leq x} \fr{\sg(n+1)}{n+1} \fr{n}{\vp(n)} =  \prod\lm_p \left( 1+\fr{2p+1}{p(p^2-1)} \right) + O(x^{\eps}).$$
 \item[(b)] $$ \sum\lm_{n \leq x} \fr{\sg(n+1)}{\vp(n)} = x \prod\lm_p \left( 1 + \fr{2p+1}{p(p^2-1)} \right) + O(x^{\eps}).$$
\end{itemize}
\end{cor}
\medskip
For comparison, we note that Stepanauskas \cite{3} has proved
$$ \sum\lm_{n\leq x} \fr{\sg(n+1)}{\vp(n)} = x\prod\lm_p \left( 1+\fr{2p+1}{p(p^2-1)} \right) + O\left( \fr{x}{(\log x)^2} \right).$$

\medskip

The method of proof of Theorem \ref{main1} can also be used to prove an asymptotic formula for 
$$\sum\lm_{p \leq x} F(p+h) G(p+k). $$
We explain this with an example $F(n)=G(n)=\fr{\vp(n)}{n}$. We prove

\begin{thm}
\label{pshift}
Fix $A>0$. Then
$$ \sum\lm_{p \leq x} \fr{\vp(p+2)}{p+2} \fr{\vp(p+1)}{p+1} = \fr{\text{li}(x)}{2} \prod\lm_{p>2} \left( 1-\fr{2}{p(p-1)} \right) + O\left( \fr{x}{(\log x)^{A-1}} \right).$$
Here the $O$-constant depends only upon $A$.
\end{thm}

\begin{rem}
The above result can be compared with Corollary 1 of \cite{4}, where the error term $O\left( \fr{\text{li}(x)}{(\log \log x)^B} \right)$ is much larger.  
\end{rem}
\medskip
\begin{proof}[Proof of Theorem \ref{pshift}]
We have
\begin{equation*}
\begin{split}
S &= \sum\lm_{p \leq x} \fr{\vp(p+2)}{p+2} \fr{\vp(p+1)}{p+1} =\sum\lm_{p \leq x} \sum\lm_{\st{a \mid p+2 \\ b \mid p+1}} \fr{\mu(a) \mu(b)}{ab}\\
&=T_1 + T_2+T_3,
\end{split}
\end{equation*}
where $T_1$ corresponds for $[a,b] \leq (\log x)^A$, $T_2$ for $(\log x)^A<[a,b] \leq x$ and $T_3$ for $[a,b]>x$. \\ \vskip 0.02in
Now, 
\begin{equation}
\label{T3}
T_3 \leq \sum\lm_{n \leq x} \sum\lm_{\st{a \mid n+2 \\ b \mid n+1\\ [a,b] \geq x}} \fr{1}{ab} = O\left( \log^2 x \right),
\end{equation}
by Lemma \ref{Lem3} (b).\vskip 0.05in
Moreover,
\begin{equation}
\label{T2}
T_2 \leq \sum\lm_{n \leq x} \sum\lm_{\st{a \mid n+2\\ b \mid n+1\\ (\log x)^A <[a,b] \leq x}} \fr{1}{ab} = \sum\lm_{\st{(a,b)=1\\(\log x)^A <[a,b] \leq x}} \fr{1}{ab} \left( \fr{x}{ab} + O(1) \right) = O\left( \fr{x}{(\log x)^{A-1}} \right).
\end{equation}
Now, 
\begin{equation}
\label{T1}
T_1 = \sum\lm_{p \leq x} \sum\lm_{\st{[a,b] \leq (\log x)^A \\ a \mid p+2 \\ b \mid p+1}} \fr{\mu(a)\mu(b)}{ab} = \sum\lm_{[a,b] \leq (\log x)^A} \fr{\mu(a)\mu(b)}{ab} \sum\lm_{\st{p \leq x \\ p \cg -2 \md{a}\\ p \cg -1 \md{b}}} 1.
\end{equation}
For $p \neq 2$, the $p$-sum survives only if $(a,b)=1$ and $a$ is odd. Thus,
$$ T_1 = \sum\lm_{\st{a \text{ odd} \geq 1 \\ (a,b)=1 \\ ab \leq (\log x)^A}} \fr{\mu(a) \mu(b)}{ab} \left( \fr{\text{li}(x)}{\vp(ab)} + O\left( \fr{x}{(\log x)^{A}} \right) \right), $$
by Siegel's theorem on primes in arithmetic progressions. Clearly, the $O$-term is $O\left( \fr{x}{(\log x)^{A-1}} \right)$. \\ \vskip 0.02in
The main term is 
$$ \text{li}(x) \sum\lm_{\st{a \text{ odd} \\ (a,b)=1}} \fr{\mu(a)\mu(b)}{ab \vp(ab)} - \text{li}(x) \sum\lm_{\st{a \text{ odd} \\ (a,b)=1 \\ ab>(\log x)^A}} \fr{\mu(a)\mu(b)}{ab \vp(ab)}. $$
The second term is $O\left( \fr{x}{(\log x)^{A-1}} \right)$ and the first term is $ \fr{\text{li}(x)}{2} \prod\lm_{p>2} \left( 1-\fr{2}{p(p-1)} \right) $.
This completes the proof.
\end{proof}

\begin{rem}
The method of proof of Theorem \ref{pshift} gives the same error term for all sums of the form 
$$\sum\lm_{p \leq x} F(p+h) G(p+k),$$ 
whenever $F=f*1$, $G=g*1$ and $f$ and $g$ are in $\mc{A}_{\al}$, $\mc{A}_{\bt}$ respectively. 
\end{rem}

\medskip

In Theorem \ref{main1}, we take $F(n)=\fr{\sg_s(n)}{n^s}$, $G(n)=\fr{\sg_t(n)}{n^t}$, with $s \leq t$, where $\sg_s(n)=\sum\lm_{d \mid n} d^s$. Then $f(n)=\fr{1}{n^s}$ and $g(n)=\fr{1}{n^t}$. 
Taking $\al=s$ and $\bt=t$ gives

\begin{cor}
Uniformly for $|h| \leq N/2$, we have
$$ \sum\lm_{n \leq N} \fr{\sg_s(n)}{n^s} \fr{\sg_t(n+h)}{(n+h)^t} = (N - H) \fr{\zeta(s+1) \zeta(t+1)}{\zeta(s+t+2)} \sg_{-(s+t+1)}(h) + O\left( E(N;s,t) \right),$$
where the $O$-term depends only on $s$ and $t$ and is independent of $h$. In particular, the error term is
$$ \begin{cases} O(N^{1-s}), & s<1 \text{ and }t>s, \\ O(N^{1-s} \log N), & s=t<1, \\ O(\log N), & 1=s<t, \\ O(\log^2 N), & s=t=1, \\ O(1), & s>1. \end{cases} $$
\end{cor}

\medskip

We can compare the above result with Corollary 1 of Coppola, Murty, Saha \cite{16}, where the error term depends on $h$, and as a function of $N$, given by
$$ \begin{cases} O(N^{1-s} (\log N)^{4-2s}), & s<1, \\ O(\log^3 N), & s=1, \\ O(1), & s>1. \end{cases} $$
Similar remarks also apply for Corollary 2 of \cite{16}.

\bigskip

\section{Proof of Theorem \tps{\ref{main2}}{\ref{main2}}}
We have,
\begin{equation*}
\begin{split}
 S &= \sum\lm_{H \leq n \leq x} \mu^2(n) G(n-h) = \sum\lm_{H \leq n \leq x} \sum\lm_{\st{a^2 \mid n \\ b \mid n-h}} \mu(a) g(b)\\
 &= T_1 + T_2,
\end{split} 
\end{equation*}
where $T_1$ corresponds to $[a^2, b] \leq x$ and $T_2$ corresponds to $[a^2,b]>x$. \\
We note that $T_2=O(x^{\eps} E_1(x))$ by Lemma \ref{Lem4} (b). \\ \vskip 0.02in
Now,
\begin{equation*}
\begin{split}
T_1 &= \sum\lm_{a,b} \mu(a) g(b) \sum\lm_{\st{n \cg 0 \md{a^2} \\ n \cg h \md{b} \\ H \leq n \leq x }} 1 = \sum\lm_{\st{[a^2,b] \leq x \\ (a^2,b) \mid h  }} \mu(a) g(b) \left( \fr{x-H}{[a^2,b]} + O(1) \right)\\
&= T_3 + T_4.
\end{split}
\end{equation*}
Now,
$$ T_3 = (x-H) \sum\lm_{(a^2,b) \mid h} \fr{\mu(a)g(b)}{[a^2,b]} + O\left( x \sum\lm_{[a^2,b] \geq x} \fr{|g(b)|}{[a^2,b]} \right). $$
In $T_3$, the main term is $(x-H) K(h)$ and the $O$-term is $O\left( E_1(x) \right)$ by Lemma \ref{Lem7} (b). \vskip 0.03in
Moreover, 
$$ T_4 = O\left( \sum\lm_{[a^2,b] \leq x} |g(b)| \right) = O\left( E_1(x) \right), $$
by Lemma \ref{Lem5} (b). This completes the proof.

\bigskip

\section{Appendix}
We now sketch how $x^{\eps}$ could be saved from the error term in Theorem \ref{main2}, if $\al$ is not in the neighbourhood of $1/2$. Let us recall the term $x^{\eps}$ occurs only in Lemma \ref{Lem4}, and so we concentrate only on this lemma. Recall that in the proof Lemma \ref{Lem4}, we have
$$ S_{A,B} = \sum\lm_{H \leq n \leq x} \sum\lm_{\st{a^2 \mid n \\ b \mid n-h \\ a \sim A \\ b \sim B}} \mu(a) b^{-\al} $$
and $$S=\sum\lm_{\st{A=2^k \leq x^{1/2} \\ B=2^l \leq x \\ A^2 B>x}} S_{A,B},$$
where $A$ and $B$ are powers of $2$ satisfying $A \leq x^{1/2}$, $B \leq x$ and $A^2 B > x$. 

\subsection*{Case I} $x^{0.05} \leq A \leq x^{0.45}$. In this case, we rewrite Lemma \ref{Lem4} as Claim 1.\\ \vskip 0.01in
\textbf{Claim 1}: $$S_{A,B} \ll \fr{x^{1+\eps}}{AB^{\al}}.$$
Summing over $A$, $B$ in this case, the sum is $O(E_1(x))$ if $\eps<0.01|1-2\al|$.
\medskip
\subsection*{Case II} $A \leq x^{0.05}$. In this case, we make \\ \vskip 0.01in
\textbf{Claim 2}: $$S_{A,B} \ll \fr{x(\log A)^{10}}{AB^{\al}}.$$
To see this, since $a^2 \mid n$, we write $n=a^2 c$.
Let 
\begin{equation} \label{T} T= \left\{ (a,b,c,d): a^2c-bd=h, \ a \sim A, b \sim B \right\} \end{equation}
Thus,
$$ S_{A,B} \ll B^{-\al} \left| T \right|. $$
Now, $bd=a^2 c-h \leq 2x$ and $a^2b>x$. hence $d \leq 2a^2 \ll x^{0.1}$. \vskip 0.01in
To count the number of elements in $T$, fix $a$ and $d$. Then the equation
$ a^2 c - h \cg 0 \md{d}$ has at most $(a^2,d)$ solutions for $c \md{d}$. Since $c \leq \fr{x}{a^2}$, the total number of choices for $c$ is at most
$$ \left( \fr{x}{a^2 d} + O(1) \right) (a^2,d). $$
Since $a \ll x^{0.05}$ and $d \ll x^{0.1}$, the $O$-term can be absorbed into the main term. Therefore,
$$|T| \ll \sum\lm_{\st{a \sim A \\ d \leq 2a^2}} \fr{x(a^2,d)}{a^2 d}, $$
and hence the claim. \vskip 0.01in
Summing over the relevant $A$ and $B$, we get the desired estimate $O(E_1(x))$.
 \medskip
\subsection*{Case III} $A \geq x^{0.45}$, $B>x^{0.2}$. In this case, we use Claim 1 and sum over the appropriate range of $A$ and $B$ and we get the upper bound $O(E_1(x))$.  
\subsection*{Case IV} $A \geq x^{0.45}$, $B \leq x^{0.2}$. In this case, we make \\ \vskip 0.01in
\textbf{Claim 3}: $$ S_{A,B} \ll \fr{x(\log B)^{10}}{AB^{\al}}. $$
As in the previous case, we need to estimate $|T|$, with $T$ as given in (\ref{T}). \vskip 0.05in
Since $a^2 c \leq x$ and $a>x^{0.45}$, it follows that $c < x^{0.1}$. We fix $c$ and $b$. Then from Lemma \ref{solutions}, the equation $a^2 c-h \cg 0 \md{b}$ has at most $L(b)\tau(b)$ solutions in $a \md{b}$. Since $a \sim A$, the total number of choices for $a$ is at most
$$ \left( \fr{A}{b} + O(1)\right) L(b) \tau(b).$$
Again, the $O$-term can be ignored. Since $a^2c \leq x$, we get $c \ll \fr{x}{A^2}$. Thus, summing the above with $c \ll \fr{x}{A^2}$ and $b \sim B$, the claim follows. \vskip 0.06in
Now, summing over $A$, $B$ in the desired range, we obtain the required bound.

\section*{Acknowledgement}
The authors would like to thank the reviewer for some usefull suggestions that helped to improve the content of this version of the paper. 

\bigskip



\begin{thebibliography}{12}

\bibitem[BG]{7} R Balasubramanian and S. Giri, {\em Mean-Value Of product Of Shifted Multiplicative functions and average number of points on Elliptic curves},  J. Number Theory \textbf{157} (2015), 37-53.
\medskip
\bibitem[Ca]{13} L. Carlitz, A note on the composition of arithmetic functions. Duke Math. J. \textbf{33} (1966), 629-632.
\medskip
\bibitem[CS]{5} E.-H. Choi, W. Schwarz, {\em Mean-values of products of shifted arithmetical functions. Analytic and probabilistic methods in number theory} (Palanga, 2001), 32-41, TEV, Vilnius, 2002.
\medskip
\bibitem[Ka]{2} I. Katai, {\em On the distribution of arithmetical functions.}, Acta Math. Acad. Sci. Hungar. \textbf{20} (1969), 69-87.
\medskip
\bibitem[Mi]{15} L. Mirsky, Summation formulae involving arithmetic functions, Duke Math. J. \textbf{16}, (1949), 261–272.
\medskip
\bibitem[CMS]{16} G. Coppola, M.R. Murty, B. Saha, {\em On the error term in a Parseval type formula in the theory of Ramanujan expansions II}, J. Number Theory \textbf{160} {2016}, 700-715.
\medskip
\bibitem[Re]{12} D. Rearick, Correlation of semi-multiplicative functions. Duke Math. J. \textbf{33} (1966), 623-627.
\medskip
\bibitem[SS]{1} J. Siaulys and G. Stepanauskas, {\em On the Mean Value of the Product of Multiplicative Functions with Shifted Argument.}, Monatsh. Math.\textbf{150}, (2007), 343-351.
\medskip
\bibitem[S1]{3} G. Stepanauskas, {\em The mean values of multiplicative functions. II}, Lithuanian Math. J. \textbf{37} (1997), 162-170.
\medskip
\bibitem[S2]{4} G. Stepanauskas, {\em The Mean Values of Multiplicative Functions on Shifted Primes}, Lithuanian Math. J. \textbf{37} (1997), 443-451.
\end{thebibliography}
\end{document}